\documentclass{amsart}

\usepackage{amssymb}
\usepackage{graphicx}
\usepackage{MnSymbol}
\usepackage[all]{xy}

\usepackage{enumerate}

\usepackage{amsmath,amscd}

\usepackage{amsthm}

\title{On homomorphic images of ultraproducts}

\author{Samuel M. Corson}
\address{E. T. S. I. I. Universidad Polit\'{e}cnica de Madrid, Jos\'{e} Guti\'{e}rrez Abascal 2, 28006 Madrid, Spain}
\email{sammyc973@gmail.com}

\theoremstyle{definition}\newtheorem{theorem}{Theorem}

\theoremstyle{definition}

\theoremstyle{definition}\newtheorem{corollary}[theorem]{Corollary}
\theoremstyle{definition}\newtheorem{proposition}[theorem]{Proposition}
\theoremstyle{definition}\newtheorem{definition}[theorem]{Definition}
\theoremstyle{definition}\newtheorem{question}[theorem]{Question}
\theoremstyle{definition}
\theoremstyle{definition}\newtheorem{remark}[theorem]{Remark}
\theoremstyle{definition}
\theoremstyle{definition}\newtheorem{lemma}[theorem]{Lemma}
\theoremstyle{definition}
\theoremstyle{definition}
\theoremstyle{definition}
\theoremstyle{definition}
\theoremstyle{definition}\newtheorem{definitions}[theorem]{Definitions}
\theoremstyle{definition}

\newcommand{\supp}{\textbf{supp}}
\newcommand{\ex}{\operatorname{ex}}
\newcommand{\Bdd}{\textbf{Bdd}}

\begin{document}

\thanks{The work of the author was supported by Basque Government Grant IT1483-22 and Spanish Government Grants PID2019-107444GA-I00 and PID2020-117281GB-I00.}

\begin{abstract}

In this short note we answer some questions of Bergman regarding homomorphic images of (ultra)products of groups.

\end{abstract}

\maketitle

\begin{section}{Introduction}

Ultrafilter constructions provide a rich source of examples and simplify many arguments.  Though such constructions nicely serve their raison d'\^{e}tre, they tend to be mysterious and unwieldy in some ways.  We consider natural questions in group theory, which derive from the following question.  If $\mathcal{U}$ is a nonprincipal ultrafilter over the set $\omega$ of natural numbers, then what can be said regarding homomorphic images of an ultraproduct $(\prod_{i \in \omega} G_i)/\mathcal{U}$?  

We focus on some specific questions of George Bergman.  For Questions \ref{Kour2013a} and \ref{Kour2013b} see \cite[Qu. 35]{B}, \cite[Qu. 20.13 (a)]{KhMaz} and respectively \cite[Qu. 20.13 (b)]{KhMaz}.  Questions \ref{localtoglobal} and \ref{isitcotorsion} are respectively \cite[Qu. 19]{B} and \cite[Qu. 33]{B}.  Question \ref{makingthesameimages} is \cite[Qu 17]{B}, \cite[Qu. 20.10 (b)]{KhMaz}.

\begin{question}\label{Kour2013a}  If an abelian group can be written as a homomorphic image of a non-principal countable ultraproduct of not necessarily abelian groups $G_i$, must it be a homomorphic image of a nonprincipal countable ultraproduct of abelian groups?
\end{question}

\begin{question}\label{Kour2013b}  If an abelian group can be written as a homomorphic image of a direct product of an infinite family of not necessarily abelian \emph{finite} groups, can it be written as a homomorphic image of a direct product of finite abelian groups?
\end{question}

\begin{question}\label{localtoglobal} If $\mathcal{U}$ is a nonprincipal ultrafilter on $\omega$ and $H$ is a group such that every $h \in H$ lies in a homomorphic image within $H$ of $(\prod_{\omega}\mathbb{Z})/\mathcal{U}$, must $H$ be a homomorphic image of a nonprincipal ultraproduct $(\prod_{i \in \omega} G_i)/\mathcal{V}$?
\end{question}

\begin{question} \label{isitcotorsion}  If an abelian group $A$ is such that every homomorphism $\bigoplus_{\omega} \mathbb{Z} \rightarrow A$ extends to a homomorphism $\prod_{\omega}\mathbb{Z} \rightarrow A$, then is $A$ cotorsion (see Definition \ref{cotorsiondef})?
\end{question}

\begin{question} \label{makingthesameimages}  If $\mathcal{U}$ and $\mathcal{U}'$ are nonprincipal ultrafilters on $\omega$, does there exist a nonprincipal ultrafilter $\mathcal{U}''$ such that every group which is the homomorphic image of an ultraproduct over $\mathcal{U}$ or over $\mathcal{U}'$ can be written as an image of an ultraproduct over $\mathcal{U}''$?
\end{question}

Questions 1 and 2 will be answered affirmatively using some new characterizations of cotorsion abelian groups (Theorem \ref{New}).  We'll see that the direct sum $\bigoplus_{n \in \omega \setminus \{0\}} \mathbb{Z}/n\mathbb{Z}$ gives a negative answer to Question \ref{localtoglobal}, and many more examples will follow from Corollary \ref{directsum}.  It was added in the proof of the article \cite{B} that Question \ref{isitcotorsion} has a positive solution; we provide a very short argument using the machinery presented.  It may not be immediately clear what Questions \ref{Kour2013b} and \ref {isitcotorsion} have to do with ultraproducts, but it will be seen that all these questions are of a kind.  The principal tool in attacking Questions 1 - 4 is the notion of Higman-completeness, due to Herefort and Hojka \cite{HH}.  We will not completely answer Question \ref{makingthesameimages}, but will show that an affirmative answer  is consistent with ZFC.  Indeed, a positive answer follows from Blass's near coherence of filters \cite{Blass}.

The paper is organized as follows.  Questions 1 - 4 are answered in Section \ref{Higmansection}.  A greater variety of examples for Question \ref{localtoglobal} is given in Section \ref{Chasesection}, and Question \ref{makingthesameimages} is treated in Section \ref{RudinKeislersection}.

\end{section}

\begin{section}{Higman-complete groups}\label{Higmansection}

We begin this section by defining a handful of concepts and state a proposition detailing some known relationships among them.  Next, we pass through some lemmas towards Theorem \ref{New}, answering some questions along the way.  We emphasize that although much of the discourse will involve abelian groups, \emph{we do not require any of our groups to be abelian unless so stated}.

\begin{definition}\label{cotorsiondef}\cite[\textsection 9.6]{F}  An abelian group $A$ is \emph{cotorsion} if every short exact sequence of abelian groups

\begin{center}

$0 \rightarrow A \rightarrow B \rightarrow C \rightarrow 0$

\end{center}

\noindent with $C$ torsion-free necessarily splits.

\end{definition}

\begin{definition}\cite[\textsection 6.1]{F}  An abelian group $A$ is \emph{algebraically compact} if there exists a compact abelian group $\mathbb{A}$ in which $A$ is a direct summand.
\end{definition}

\begin{definition}\cite{HH}  A group $G$ is \emph{Higman-complete} if for each sequence $(f_i)_{i \in \omega}$ of elements in $G$ and sequence of free words in two variables $(w_i)_{i \in \omega}$ there exists a sequence $(h_i)_{i \in \omega}$ of elements in $G$ such that $h_i = w_i(f_i, h_{i + 1})$ for each $i \in \omega$.
\end{definition}

\begin{remark}\cite[Lemma 2]{HH}\label{homHig}  The class of Higman-complete groups is closed under taking homomorphic images.  
\end{remark}

The following is known.

\begin{proposition}\label{Known}  Let $A$ be an abelian group.  The following are equivalent.

\begin{enumerate}[(i)]

\item $A$ is cotorsion.

\item $A$ is a homomorphic image of an algebraically compact group.

\item $A$ is Higman-complete.

\item $A$ is the homomorphic image of an ultraproduct $(\prod_{i \in \omega} A_i)/\mathcal{U}$, where $\mathcal{U}$ is nonprincipal and the $A_i$ are abelian.

\end{enumerate}
\end{proposition}

\begin{proof} 
The equivalence of (i) and (ii) is classically known.  The equivalence of (i) and (iii) is \cite[Theorem 3]{HH}.  The equivalence of (i) and (iv) is in \cite[Proposition 23]{B}.
\end{proof}

With this information we are already equipped for the following.

\begin{proof}[Proof of affirmative solution to Question \ref{isitcotorsion}]  Assume the hypotheses.  We will show that $A$ is Higman-complete, and the fact that $A$ is cotorsion then follows immediately from Proposition \ref{Known}.  Let $(f_i)_{i \in \omega}$ be a sequence of elements of $A$ and $(w_i)_{i \in \omega}$ a sequence of words in two variables.  Since $A$ is abelian, we can without loss of generality assume that the word $w_i(x, y)$ is of form $z_ix + z_i'y$ where $z_i, z_i' \in \mathbb{Z}$.  For each $i \in \omega$, let $\delta_i \in \mathbb{Z}^{\omega}$ be the element with $\delta_i(i) = 1$ and $\delta_i(j) = 0$ for $j \neq i$.  Letting $\phi: \bigoplus_{\omega} \mathbb{Z} \rightarrow A$ be defined by the function $\delta_i \mapsto f_i$, we obtain (by the hypothesis
of Question 4) an extension $\Phi: \mathbb{Z}^{\omega} \rightarrow A$.  Letting

\begin{itemize}

\item $h_0 = \Phi(z_0, z_0'z_1, z_0'z_1'z_2, z_0'z_1'z_2'z_3, \ldots)$

\item $h_1 = \Phi(0, z_1, z_1'z_2, z_1'z_2'z_3, \ldots)$

\item $h_2 = \Phi(0, 0, z_2, z_2'z_3, z_2'z_3'z_4, \ldots)$

$\vdots$

\end{itemize}

\noindent it is easy to check that the sequence $(h_i)_{i\in \omega}$ satisfies $h_i = w_i(f_i, h_{i + 1})$.  For example, $$h_0 = \Phi(z_0\delta_0) + z_0'\Phi(0, z_1, z_1'z_2, \ldots) = z_0f_0 + z_0'h_1$$ and for $i > 0$ the check is entirely analogous.  Thus $A$ is Higman-complete, and therefore cotorsion by Proposition \ref{Known}.

\end{proof}

\begin{lemma}\label{prodmodsum}  For an arbitrary sequence of groups $(G_i)_{i \in \omega}$ the group $$G = (\prod_{i \in \omega} G_i)/(\bigoplus_{i \in \omega} G_i)$$ is Higman complete.  Also, for any nonprincipal ultrafilter $\mathcal{U}$ the ultraproduct $(\prod_{i \in \omega} G_i)/\mathcal{U}$ is Higman-complete.
\end{lemma}

\begin{proof}
In this lemma we will use symbols $\zeta$ and $\eta$ to denote elements of $\prod_{i \in \omega} G_i$ and $[\zeta], [\eta]$ to denote elements of $G$.  Let a sequence $([\zeta_\ell])_{\ell \in \omega}$ of elements of $G$ be given, together with a sequence $(w_{\ell})_{\ell \in \omega}$ of free words in two variables.  For each $\ell \in \omega$ we fix a representative $\zeta_{\ell} \in [\zeta_{\ell}]$.  We shall define by induction a sequence $(\eta_{\ell})_{\ell \in \omega}$ of elements in $\prod_{i \in \omega} G_i$.  Let $\eta_0(0) = w_0(\zeta_0(0), 1_{G_0})$.  Let $\eta_1(1) = w_1(\zeta_1(1), 1_{G_1})$ and $\eta_0(1) = w_0(\zeta_0(1), \eta_1(1))$.  Generally, if $\eta_{\ell}(j)$ has been defined for all $\ell \leq j \leq k$, then we let

\begin{itemize}
\item $\eta_{k + 1}(k + 1) = w_{k + 1}(\zeta_{k + 1}(k + 1), 1_{G_{k + 1}})$

\item $\eta_k(k + 1) = w_k(\zeta_k(k + 1), \eta_{k + 1}(k + 1))$

\item $\eta_{k - 1}(k + 1) = w_{k  - 1}(\zeta_{k - 1}(k + 1), \eta_k(k + 1))$

$\vdots$

\item $\eta_0(k + 1) = w_0(\zeta_0(k + 1), \eta_1(k + 1))$.

\end{itemize}

\noindent For $j <  \ell$ we let $\eta_{\ell}(j) = 1_{G_j}$.

We claim that the sequence $([\eta_{\ell}])_{\ell \in \omega}$ of elements in $G$ satisfies the definition of Higman-completeness.  To see this, let $\ell \in \omega$ be given.  By construction we have for all $j \geq \ell$ that $\eta_{\ell}(j) = w_{\ell}(\zeta_{\ell}(j), \eta_{\ell + 1}(j))$, so that we can indeed write $[\eta_{\ell}] = w_{\ell}([\zeta_{\ell}], [\eta_{\ell + 1}])$.

For the claim in the second sentence of this lemma we note that $(\prod_{i \in \omega} G_i)/\mathcal{U}$ is a quotient of $(\prod_{i \in \omega} G_i)/(\bigoplus_{i \in \omega} G_i)$ and apply Remark \ref{homHig}.

\end{proof}

\begin{proof}[Proof of negative answer to Question \ref{localtoglobal}]  Consider the abelian group $$H = \bigoplus_{n \in \omega} C_n$$ where the $C_n$ are finite cyclic groups of unbounded orders. This group is torsion, and so each $h \in H$ is an element of a finite cyclic subgroup $C$ of $H$, say $|C| = m$, and it a trivial matter to surject $(\prod_{\omega}\mathbb{Z})/\mathcal{U}$ onto $(\prod_{\omega}\mathbb{Z}/m\mathbb{Z})/\mathcal{U}$ and as $m$ is finite we have $(\prod_{\omega}\mathbb{Z}/m\mathbb{Z})/\mathcal{U}$ isomorphic to $C$.  Suppose by way of contradiction that $H$ is a homomorphic image of a nonprincipal ultraproduct $(\prod_{i \in \omega} G_i)/\mathcal{V}$.  Then $H$ is Higman-complete (by Lemma \ref{prodmodsum} and Remark \ref{homHig}).  Thus $H$ is an abelian group which is Higman-complete, and so $H$ is cotorsion (Proposition \ref{Known}).  Since $H$ is a torsion cotorsion group, it is a direct sum of a bounded group (i.e. a group in which $a^n = 0$ for some universal $n \in \omega$) and a divisible group (\cite[Corollary 9.8.4]{F}).  However $H$ has no nontrivial divisible elements and is not bounded, a contradiction.  We shall later see that many nonabelian groups of this flavor also serve as counterexamples.
\end{proof}

We interrupt at this stage to give a result involving elementary equivalence, which may be of independent interest.

\begin{proposition}\label{cotorsionimage}  Let $G$ be a group.  The following are equivalent.

\begin{enumerate}[(a)]

\item  There exists a group $H$ which is elementarily equivalent to $G$ such that the abelianization $H/H'$ is nontrivial cotorsion.

\item  Either $G$ is not perfect or $G$ has elements of unbounded commutator length.

\end{enumerate}

\end{proposition}

\begin{proof} (b) $\Rightarrow$ (a).  Take $\mathcal{U}$ to be a nonprincipal ultrafilter on $\omega$.  Let $H$ be the ultrapower $(\prod_{\omega} G)/\mathcal{U}$.  Then $H$ is Higman-complete (Lemma \ref{prodmodsum}), elementarily equivalent to $G$, and the abelianization $H/H'$ is cotorsion.  It remains to prove that $H/H'$ is nontrivial.  If $G$ is not perfect, then the natural surjective map from $H = (\prod_{\omega}G)/\mathcal{U}$ to $(\prod_{\omega}G/G')/\mathcal{U}$ gives a nontrivial abelian image of $H$, so $H/H'$ is nontrivial.  On the other hand if $G$ is perfect then $G$ has elements of arbitrarily long commutator length, so select an element $(g_i)_{i \in \omega}$ in $\prod_{i \in \omega} G$ such that the commutator length of $g_i$ is greater than $i$.  Then the element $[(g_i)_{i \in \omega}]$ in $H$ is not in the kernel of the abelianization of $H$ (since its commutator length is infinite).

\noindent (a) $\Rightarrow$ (b) We prove the contrapositive.  Supposing that $G$ is perfect and of bounded commutator length, say length $j$, every group $H$ which is elementarily equivalent to $G$ also is perfect of bounded commutator length at most $j$ (since this property can be expressed using a first-order formula).  Then the abelianization of such an $H$ is trivial.
\end{proof}

Continuing our progress toward Theorem \ref{New} we give two more lemmas.

\begin{lemma}\label{product}  If $\{G_j\}_{j \in J}$ is a collection of Higman-complete groups then the product $$G = \prod_{j \in J} G_j$$ is Higman-complete.
\end{lemma}

\begin{proof}
Given a sequence $(\zeta_i)_{i}$ of elements in $G$, and a sequence $(w_i)_{i \in \omega}$ of free words in two variables, we solve the necessary equations coordinatewise.  More precisely, for a fixed $j \in J$ we have a sequence $(\zeta_i(j))_{i \in \omega}$ of elements of $G_j$ and we select $(\eta_i(j))_{i \in \omega}$ so that $\eta_i(j) = w_i(\zeta_i(j), \eta_{i + 1}(j))$ for each $i \in \omega$ (which is possible since $G_j$ is Higman-complete).  Now, it is clear that $\eta_i = w_i(\zeta_i, \eta_{i + 1})$ holds, since the equation holds in each coordinate.
\end{proof}

\begin{definition}\cite[\textsection 4.1]{F}  An abelian group is \emph{reduced} if it has no nontrivial divisible elements.
\end{definition}

Lemma \ref{divisibleimage} below is immediate from \cite[Theorem 4.3.1]{F}, and Lemma \ref{algebraicallycompact} is probably known.

\begin{lemma}\label{divisibleimage}  An abelian group $D$ is divisible if and only if $D$ is the homomorphic image of $\bigoplus_{\kappa}\mathbb{Q}$ for some cardinal $\kappa$.
\end{lemma}

\begin{lemma}\label{algebraicallycompact}  If $\mathbb{A}$ is an algebraically compact group, then $\mathbb{A}$ is a homomorphic image of a product of finite cyclic groups.
\end{lemma}

\begin{proof}
Letting $D$ denote the maximal divisible subgroup of $\mathbb{A}$, we can write $\mathbb{A} = D \oplus R$ where $R$ is a reduced algebraically compact group.  A reduced algebraically compact group is a direct summand of a product of cyclic groups of prime power order \cite[Corollary 6.1.4]{F}, so in particular $R$ is a homomorphic image of a product of finite cyclic groups.

Letting $\mathbb{P}$ denote the set of prime natural numbers, it is easy to see that $B = (\prod_{p \in \mathbb{P}} \mathbb{Z}/p\mathbb{Z})/(\bigoplus_{p \in \mathbb{P}} \mathbb{Z}/p\mathbb{Z})$ is torsion-free, divisible, and infinite.  Thus $B$ is a vector space over the field $\mathbb{Q}$ of dimension greater than $0$, so by picking a basis and projecting to a single coordinate we obtain a homomorphism from $B$ onto $\mathbb{Q}$.  In particular, $\mathbb{Q}$ is a homomorphic image of a product of finite cyclic groups.

By Lemma \ref{divisibleimage} take $\kappa$ to be a cardinal such that $D$ is a homomorphic image of $\bigoplus_{\kappa}\mathbb{Q}$.  Now $\bigoplus_{\kappa}\mathbb{Q}$ is a homomorphic image of $\mathbb{Q}^{\kappa}$ (by a vector space argument), and $\mathbb{Q}^{\kappa}$ is a homomorphic image of a product of finite cyclic groups (since $\mathbb{Q}$ is such an image).  Thus $D$ is a homomorphic image of a product of finite cyclic groups, and as $R$ also satisfies this property, the group $\mathbb{A} = D \oplus R$ also has this property.

\end{proof}

\begin{theorem}\label{New}  Let $A$ be an abelian group.  The following are equivalent.

\begin{enumerate}

\item $A$ is cotorsion.

\item $A$ is a homomorphic image of $(\prod_{i \in \omega} G_i)/(\bigoplus_{i \in \omega} G_i)$ for some sequence $(G_i)_{i \in \omega}$ of groups which are not necessarily abelian.

\item $A$ is a homomorphic image of a product of finite cyclic groups.

\item $A$ is a homomorphic image of a product $\prod_{j \in J} F_j$ where each $F_j$ is a finite group which is not necessarily abelian.

\end{enumerate}

\end{theorem}

\begin{proof}

We will first prove the equivalence of (1) and (2).   Suppose that $A$ is cotorsion.  By Proposition \ref{Known} we know that there is a sequence $(A_i)_{i \in \omega}$ of abelian groups and a nonprincipal ultrafilter $\mathcal{U}$ on $\omega$ so that $A$ is a homomorphic image of the ultraproduct $(\prod_{i \in \omega} A_i)/\mathcal{U}$.  As $(\prod_{i \in \omega} A_i)/\mathcal{U}$ is a homomorphic image of $(\prod_{i \in \omega} A_i)/(\bigoplus_{i \in \omega} A_i)$, so is $A$.  For the other direction, if $A$ is a homomorphic image of a group of form $(\prod_{i \in \omega} G_i)/(\bigoplus_{i \in \omega} G_i)$ then $A$ is Higman-complete (Lemma \ref{prodmodsum}), so $A$ is cotorsion (Proposition \ref{Known}).

Now we argue that (1) $\Rightarrow$ (3) $\Rightarrow$ (4) $\Rightarrow$ (1).  That (1) implies (3) follows from the fact that a cotorsion group is a homomorphic image of an algebraically compact group (Proposition \ref{Known}) and the fact that an algebraically compact abelian group is a homomorphic image of a product of finite cyclic groups (Lemma \ref{algebraicallycompact}).  That (3) implies (4) is evident.  For the last implication, note that a finite group $F$ is Higman-complete.  To see this, take a nonprincipal ultrafilter $\mathcal{V}$ on $\omega$ and recall that $F$ is isomorphic to the ultrapower $(\prod_{\omega} F)/\mathcal{V}$, so $F$ is Higman-complete (Proposition \ref{prodmodsum}).  Therefore a product of finite groups is Higman-complete (Lemma \ref{product}), and so $A$ is Higman-complete as the homomorphic image of a Higman-complete group (Remark \ref{homHig}).  Thus $A$ is cotorsion (Proposition \ref{Known}).

\end{proof}

Now it is clear that Question \ref{Kour2013a} has an affirmative solution (using (2) $\Rightarrow$ (1) of Theorem \ref{New} and (i) $\Rightarrow$ (iv) of Proposition \ref{Known}) as does Question \ref{Kour2013b} (using (4) $\Rightarrow$ (3)).

\end{section}

\begin{section}{An analogue of a theorem of Chase}\label{Chasesection}

In the light of the very abelian nature of many of the previous results, we give a general method for producing (nonabelian) examples which give a negative solution to Question \ref{localtoglobal}.  These will be torsion groups which are not the homomorphic images of any nonprincipal ultraproduct $(\prod_{i \in \omega} G_i)/\mathcal{U}$.  For this we will prove an analogue of a result of Chase \cite{Ch}.

\begin{definitions}  We'll say a group $H$ is a \emph{bounded torsion group} if there exists some $n \in \omega$ such that $(\forall h \in H) h^n = 1_H$.  For such an $H$ we let $\ex(H) \in \omega$ denote the smallest such $n$.  If a torsion group is not bounded then we shall say it is an \emph{unbounded torsion group}.
\end{definitions}

\begin{definitions}  Suppose $\{H_j\}_{j \in J}$ is a collection of bounded torsion groups.  For $h = (h_j)_{j \in J} \in \prod_{j \in J} H_j$ we let $\supp(h) = \{j \in J \mid h_j \neq 1_{H_j}\}$.  Define $$\Bdd\{H_j\}_{j \in J} = \{h \in \prod_{j \in J} H_j \mid (\exists n \in \omega)(\forall j \in \supp(h)) \ex(H_j) \leq n\}.$$
\end{definitions}

\begin{remark}  It is clear that $\Bdd\{H_j\}_{j \in J}$ is a subgroup of $\prod_{j \in J} H_j$, since if $n$ and $n'$ respectively witness that $h$ and $h'$ are in $\Bdd\{H_j\}_{j \in J}$ then $\max(n, n')$ witnesses that $hh' \in \Bdd\{H_j\}_{j \in J}$.  Moreover, this subgroup is torsion.
\end{remark}

\begin{theorem}\label{Chaseanalogue}  Let $G$ be a topological group which is either

\begin{enumerate}

\item completely metrizable; or

\item Hausdorff and locally countably compact

\end{enumerate}

\noindent and each $H_j$ a bounded torsion group and let $\phi: G \rightarrow \Bdd\{H_j\}_{j \in J}$ be an abstract group homomorphism.  Then there is some open neighborhood $U$ of the identity $1_G$ and $N \in \omega \setminus \{0\}$ such that $(\forall h \in \phi(U)) h^N = 1$.
\end{theorem}

\begin{proof}
Let $H = \Bdd\{H_j\}_{j \in J}$.  Suppose by way of contradiction that the conclusion fails.  The strategy will be to produce a system of equations (like in the definition of Higman-completeness) whose solution will give a contradiction.  Suppose we are in situation (1) and let $d$ be a complete metric which induces the topology on $G$.  We inductively choose

\begin{enumerate}[(a)]

\item a sequence $(g_m)_{m \in \omega}$ of elements in $G$; and

\item a sequence $(r_m)_{m \in \omega}$ of positive natural numbers.

\end{enumerate}

\noindent Select $g_0 \in G$ such that $\phi(g_0)$ has order greater than $1 = 0!$.  Let $r_0$ be the least common multiple of the finite nonempty set $\{\ex(H_j)|j \in \supp(\phi(g_0))\}$.  Assume we have chosen $g_0, \ldots, g_m$ and $r_0, \ldots, r_m$.  Select neighborhood $U$ of $1_G$ so that $g \in U$ implies 

\begin{itemize}

\item $d(g_mg^{r_m}, g_m) < \frac{1}{3}$;

\item $d(g_{m-1}(g_mg^{r_m})^{r_{m - 1}}, g_{m - 1}(g_m)^{r_{m-1}}) < \frac{1}{3^2}$;

$\vdots$

\item $d(g_0(g_1(\cdots (g_{m - 1} g^{r_m})^{r_{m - 1}}\cdots )^{r_1})^{r_0}, g_0(g_1(\cdots (g_m)^{r_{m - 1}} \cdots)^{r_1})^{r_0}) < \frac{1}{3^m }$.

\end{itemize}

\noindent Select $g_{m + 1} \in U$ such that $\phi(g_{m + 1})$ has order greater than $$((m + 1)r_0\cdots r_m)!\cdot r_0\cdot r_1 \cdots  \cdot r_m.$$  Let $r_{m + 1}$ be the least common multiple of the finite nonempty set $\{\ex(H_j)| j \in \supp(\phi(g_{m + 1}))\}$.  Now our sequences are defined.

For $m \leq k \in \omega$ define the element $$g_{m, k} = g_m(\cdots (g_{k - 1}(g_k)^{r_{k - 1}} )^{r_{k - 2}}\cdots)^{r_m}.$$  For $m \leq s < t \in \omega$ we have by construction

$$
\begin{array}{ll}
d(g_{m, s},g_{m, t}) & \leq d(g_{m, s}, g_{m, s + 1})+d(g_{m, s + 1}, g_{m, s + 2}) + \cdots +d(g_{m, t - 1}, g_{m, t}) \vspace*{2mm}\\
& < \frac{1}{3^{s + 1 - m}} + \frac{1}{3^{s+2 - m}} + \cdots + \frac{1}{3^{t - m}}
\end{array}
$$

\noindent In particular the sequence $g_{m, m}, g_{m, m + 1}, g_{m, m + 1}, \ldots$ is Cauchy and converges to, say, $g_{m, \infty}$.  By continuity we see that for each $m \in \omega$ the equation $$g_{m, \infty} = g_mg_{m + 1, \infty}^{r_m}$$ holds.

We know that $\phi(g_{0, \infty})$ has finite order and so let $N \in \omega \setminus \{0\}$ be such that $(\phi(g_{0, \infty}))^N = 1_H$.  As $(\phi(g_N))^{r_0 \cdots r_{N - 1}}$ has order greater than $(N \cdot r_0 \cdots r_{N - 1})!$, select $j_0 \in \supp(g_N)$ such that $\pi_{j_0}(\phi(g_N))^{r_0 \cdots r_{N - 1}}$ has order greater than $N \cdot r_0 \cdots r_{N - 1}$, where $\pi_{j_0}: \prod_{j \in J} H_j \rightarrow H_{j_0}$ is projection.  Clearly $\pi_{j_0} \circ \phi(g_N(g_{N+1, \infty})^{r_N})) = \pi_{j_0}\circ\phi(g_N)$ since $\ex(H_{j_0})$ divides $r_N$.  If $j_0 \in \bigcup_{\ell = 0}^{N - 1}\supp(g_{\ell})$ then $\ex(H_{j_0})$ divides $r_k$ for some $0 \leq k \leq N-1$, so in particular $\pi_{j_0}\circ \phi(g_N^{r_0 \cdots r_{N-1}}) = 1_{H_{j_0}}$, which is a contradiction.  Therefore $$\pi_{j_0}\circ \phi(g_{0, \infty}) = \pi_{j_0} \circ \phi(g_0(g_1(\cdots g_{N-1}(g_Ng_{N+1, \infty}^{r_{N}})^{r_{N-1}} \cdots)^{r_1})^{r_0}) = \pi_{j_0}\circ \phi(g_N^{r_0 \cdots r_{N-1}})$$ has order greater than $N \cdot r_0 \cdots r_{N - 1} > N$, while $\phi(g_{0, \infty})$ had order $N$, contradiction.  

In situation (2) we construct

\begin{enumerate}[(i)]

\item a sequence $(g_m)_{m \in \omega}$ of elements in $G$;

\item a sequence $(r_m)_{m \in \omega}$ of positive natural numbers; and

\item a sequence $(V_m)_{m \in \omega}$ of open neighborhoods of $1_G$.

\end{enumerate}

Pick $g_0 \in G$ such that $\phi(g_0)$ has order greater than $1$, take $r_0$ to be the least common multiple of $\{\ex(H_j)|j \in \supp(\phi(g_0))\}$, and let $V_0$ be an open neighborhood of $1_G$ such that $\overline{V_0}$ is countably compact.  Assume we have made the selections for all subscripts less than or equal to $m$.  Select $g_{m + 1} \in V_m$ such that $\phi(g_{m + 1})$ has order greater than $((m + 1)r_0\cdots r_m)!\cdot r_0\cdot r_1 \cdots  \cdot r_m$.  Let $r_{m + 1}$ be the least common multiple of $\{\ex(H_j)| j \in \supp(\phi(g_{m + 1}))\}$.  Let $V_{m + 1}$ be a neighborhood of $1_G$ such that $$g_{m + 1}V_{m + 1}^{r_{m + 1}} \subseteq V_m.$$

Note that $g_{m + 1}(\overline{V_{m + 1}})^{r_{m + 1}} \subseteq \overline{V_m}$ for each $m \in \omega$.  Letting $$Y_m = g_0(g_1(g_2(\cdots g_m(\overline{V_m})^{r_m}  \cdots)^{r_2})^{r_1})^{r_0}$$ for $m \in \omega$, we have $Y_m \supseteq Y_{m + 1}$ and $Y_m$ is countably compact.  Select $g_{\infty} \in \bigcap_{m \in \omega} Y_m$ and let $N$ be the order of $\phi(g_{\infty})$.  Select $\underline{g} \in \overline{V_N}$ for which $$g_{\infty} = g_0(g_1(g_2(\cdots g_N(\underline{g})^{r_N}\cdots)^{r_2})^{r_1})^{r_0}.$$  Select $j_0$ and derive the same contradiction as in situation (1).
\end{proof}

\begin{corollary}\label{intheproduct}  Suppose $\{G_i\}_{i \in \omega}$ is a collection of groups, $\{H_j\}_{j \in J}$ is a collection of bounded torsion groups, and $\mathcal{U}$ is a nonprincipal ultrafilter on $\omega$.  Then the image of any abstract homomorphism $$\phi: (\prod_{\omega} G_i)/\mathcal{U} \rightarrow \Bdd\{H_j\}_{j \in J}$$ is a bounded torsion group.
\end{corollary}

\begin{proof}
Assume the hypotheses.  By considering each group $G_i$ as a discrete topological group, it is well-known that the topological group $\prod_{i \in \omega} G_i$ is completely metrizable, and the subgroups of form $\{1_{G_0}\}\times \cdots \times \{1_{G_M}\} \times \prod_{i > M} G_i$ give a basis of neighborhoods for identity.  Applying Theorem \ref{Chaseanalogue} we have some $M, N \in \omega$ for which $\phi \circ \rho(\{1_{G_0}\}\times \cdots \times \{1_{G_M}\} \times \prod_{i > M} G_i)$ consists of elements of order dividing $N$, where $\rho: \prod_{\omega} G_i \rightarrow (\prod_{\omega} G_i)/\mathcal{U}$ is the natural map.    Since the restriction $\rho \upharpoonright \{1_{G_0}\}\times \cdots \times \{1_{G_M}\} \times \prod_{i > M} G_i$ is surjective onto the ultraproduct, we are done.
\end{proof}

For convenience, in the remainder of this section we'll say a group which is a homomorphic image of a nonprincipal ultraproduct over $\omega$ is a \emph{ui-group} (``ui'' for ultraproduct image).

\begin{corollary}\label{directsum}  If $\{H_j\}_{j \in J}$ is a collection of bounded torsion groups for which $\{\ex(H_j)\}_{j \in J}$ is unbounded in $\omega$, then the direct sum $\bigoplus_{j \in J} H_j$ is not a ui-group.
\end{corollary}

Part of the attraction of Corollary \ref{directsum} is that it provides many examples of non-ui-groups which are (unbounded) torsion.  One can apply Corollary \ref{directsum} to the collection $\{S_j\}_{j \in \omega}$ where $S_j$ is the symmetric group on $j$ elements, or to the collection $\{\mathfrak{b}(2, j)\}_{j \in \omega \setminus \{0, 1\}}$ where $\mathfrak{b}(2, j)$ is the free Burnside group of rank $2$ and exponent $j$.

There exist ui-groups which are unbounded torsion.  For example, a quasicyclic group $\mathbb{Z}(p^{\infty})$ is cotorsion and therefore a ui-group.  We do not know the answers to the following.

\vspace{.25cm}

\noindent \textbf{Questions.}

\begin{enumerate}[(A)] \item If a ui-group is unbounded torsion then does it include a nontrivial divisible subgroup?

\item Is there a bounded torsion group which is not a ui-group?  \end{enumerate}

\vspace{.25cm}

\noindent A classical open problem asks whether there exists a compact unbounded torsion group $C$ \cite[Qu. 17.93]{KhMaz}.  Such a $C$ would be profinite \cite[Theorem 28.20]{HeRo} and therefore any divisible subgroup would be trivial.  Moreover for any nonprincipal ultrafilter $\mathcal{U}$ on $\omega$ it is easy to construct a homomorphism from the ultrapower $(\prod_{\omega} C)/\mathcal{U}$ onto $C$.  Thus $C$ would give a negative answer to Question (A).

We give a few more examples of non-ui-groups.  Recall that a group $H$ is \emph{cm-slender} if every abstract group homomorphism from a completely metrizable group to $H$ has open kernel \cite{ConCor}.  A cm-slender group is necessarily torsion-free, and examples include Baumslag-Solitar groups and nontrivial word-hyperbolic groups.  Arguing as in Corollary \ref{intheproduct} it is clear that a group having nontrivial homomorphic image in a cm-slender group is not a ui-group.  Reasoning similarly, an infinite subgroup of the mapping class group of a connected compact surface is not a ui-group \cite[Theorem 9.1]{BoCo}, nor is a group which is nontrivially a free product of groups \cite[Theorem 1.5]{Sl}.

\begin{remark}  By making easy changes to the proof of Theorem \ref{Chaseanalogue} one can show an analogous statement where $G$ is instead the fundamental group $\pi_1(E)$ of the infinite earring $E$ \cite{CanCon}.
\end{remark}

\end{section}

\begin{section}{A common element below two ultrafilters} \label{RudinKeislersection}

In this brief section we remind the reader of the Rudin-Keisler ordering (see e.g. \cite[Ch. 11]{Halbeisen}) and point out why Question \ref{makingthesameimages} has a consistent affirmative answer.

\begin{definition}  For ultrafilters $\mathcal{U}$ and $\mathcal{U}'$ on $\omega$ we write $\mathcal{U} \leq_{RK} \mathcal{U}'$ if there exists a function $f: \omega \rightarrow \omega$ such that $f(\mathcal{U}') := \{X \subseteq \omega \mid (\exists Y \in \mathcal{U}') X \supseteq f(Y)\} = \mathcal{U}$.
\end{definition}

\begin{proposition}\label{RKprop}  Suppose that $\mathcal{U} \leq_{RK} \mathcal{U}'$.  Then every homomorphic image of an ultraproduct over $\mathcal{U}'$ is also a homomorphic image of an ultraproduct over $\mathcal{U}$.
\end{proposition}

\begin{proof}  Assume the hypotheses.  Let $(G_i)_{i \in \omega}$ be a sequence of groups.  Take $f: \omega \rightarrow \omega$ such that $f(\mathcal{U}') = \mathcal{U}$.  For each $j \in \omega$ in the image of $f$ select a cardinal $\kappa_j$ large enough that there is a homomorphic surjection $\phi_j: F_{\kappa_j} \rightarrow \prod_{i \in f^{-1}(\{j\})} G_i$ from the free group $F_{\kappa_j}$ of rank $\kappa_j$.  For $j \in \omega$ which is not in the image of $f$ we let $\kappa_j = 0$.  Now we obtain a composition $\Psi = \rho \circ \Phi$ of homomorphisms

$$\xymatrix{\prod_{j \in \omega} F_{\kappa_j} \ar[r]^{\Phi} & \prod_{i \in \omega}G_i \ar[r]^{\rho} & (\prod_{i \in \omega} G_i) /\mathcal{U}'}$$

\noindent which are evidently surjective ($\Phi$ is defined componentwise by the $\phi_j$).  Suppose that $(W_j)_{j \in \omega} \in \prod_{\omega} F_{\kappa_j}$ is such that $W_j$ is identity for almost every $j$, say for all elements $j$ in $X \in \mathcal{U}$.  Then for all $j \in X$ we have $\phi_j(W_j)$ is identity, and as $f^{-1}(X) \in \mathcal{U}'$, we have $(W_j)_{j \in \omega}$ in the kernel of $\Psi$.  So, $\Psi$ descends to a homomorphic surjection $\overline{\Psi}: (\prod_{j \in \omega} F_{\kappa_j})/\mathcal{U} \rightarrow (\prod_{i \in \omega} G_i) /\mathcal{U}'$.  Thus every homomorphic image of $(\prod_{i \in \omega} G_i)/\mathcal{U}'$ is also a homomorphic image of $(\prod_{j \in \omega} F_{\kappa_j})/\mathcal{U}$.
\end{proof}

The assertion \emph{Near coherence of filters} (NCF) states that for any two nonprincipal ultrafilters $\mathcal{U}$ and $\mathcal{U}'$ there exists a finite-to-one function $f: \omega \rightarrow \omega$ such that $f(\mathcal{U}) = f(\mathcal{U}')$ \cite[\textsection 5]{Blass}.  In particular, for any two nonprincipal ultrafilters $\mathcal{U}$ and $\mathcal{U}'$, NCF gives a nonprincipal ultrafilter $\mathcal{U}''$ which is $\leq_{RK}$ below both $\mathcal{U}$ and $\mathcal{U}'$.  By Proposition \ref{RKprop} the homomorphic images of ultraproducts over $\mathcal{U}$ or $\mathcal{U}'$ are also homomorphic images of ultraproducts over $\mathcal{U}''$.  Since NCF is consistent with ZFC \cite{BlassShelah}, Question \ref{makingthesameimages} has a consistent positive answer.

\end{section}

\section*{Acknowledgement}

The author thanks the referees for pointing out clarifications and typos in a very timely manner.

\end{document}